\documentclass{amsart}
\usepackage{colonequals,enumitem,amssymb,amsmath}
\usepackage{graphicx,caption,subcaption}
\usepackage{multirow}

\theoremstyle{plain}
\newtheorem{thm}{Theorem}[section]

\newtheorem{conj}[thm]{Conjecture}

\newtheorem*{defn}{Definition}

\title{Wickets in 3-uniform Hypergraphs}

\begin{document}

\begin{abstract}
In these notes, we consider a Tur\'an-type problem in hypergraphs. What is the maximum number of edges if we forbid a subgraph?
Let $H_n^{(3)}$ be a 3-uniform linear hypergraph, i.e. any two edges have at most one vertex common. A special hypergraph, called {\em wicket}, is formed by three rows and two columns of a $3 \times 3$ point matrix. We describe two linear hypergraphs -- both containing a wicket -- that if we forbid either of them in $H_n^{(3)}$, then the hypergraph is sparse, and the number of its edges is $o(n^2)$. This proves a conjecture of Gy\'arf\'as and S\'ark\"ozy.
\end{abstract}

\author{Jozsef Solymosi}
\address{Department of Mathematics, University of British Columbia, Vancouver, Canada, and Obuda University, Budapest, Hungary}
\email{solymosi@math.ubc.ca}

\maketitle

\section{Introduction}

In this note, we answer a question of Gy\'arf\'as and S\'ark\"ozy. To formulate it, let's review some definitions, notations and a brief history of the problem. 

\begin{defn}
A hypergraph is linear if two edges intersect at most one vertex.
\end{defn}

We will work with 3-uniform linear hypergraphs, i.e. hypergraphs where every edge has three vertices. 
We use the notations $H_n^3$ and $F_n^3$ for 3-uniform hypergraphs on $n$ vertices and $G_n$ for graphs on $n$ vertices. The number of edges in $H$ is denoted by $e(H)$.

It is a notoriously difficult problem to characterize unavoidable subgraphs in dense 3-uniform linear hypergraphs. One of the central conjectures in extremal combinatorics is the following:

\begin{conj}[Brown, Erd\H os, and S\'os]\label{BES}
For every $\ell \geq 3$ and $c>0$ there exists an $n_0 = n_0(c)$ such that if $n>n_0$ and $e(H_n^3) \geq cn^2$
then there exists a $F_{\ell+3}^3 \subseteq H_n^3$ such that $e(F_{\ell+3}^3) \geq \ell$.
\end{conj}

Although the conjecture does not specify that the hypergraph is linear, it is enough to consider it for linear hypergraphs. 

The $\ell=3$ case of Conjecture \ref{BES} was proved by Ruzsa and Szemer\'edi \cite{RSz}, but for larger values, the conjecture is still open. We refer to works in \cite{SaSe, SoSo, Conat} for partial results and related bounds.

In a closely related line of research, the following question was raised. For a given 3-uniform linear hypergraph, $F$, what is the maximum number of edges in a 3-uniform linear hypergraph without $F$ as a subgraph? 

\begin{defn}
    The Tur\'an number of a linear 3-uniform hypergraph $F$, denoted by $ex_L(n,F)$ is the maximum number of edges of a 3-uniform linear hypergraph not containing a subgraph isomorphic to $F$.
\end{defn}

In \cite{GyS_1} Gy\'arf\'as and S\'ark\"ozy investigated $ex_L(n,F)$ for any $F$ with at most five edges. They had good estimates with the exception of one configuration, which they called {\em wicket}. The wicket, denoted by $W$, is formed by three rows and two columns of a $3\times 3$ point matrix (fig. \ref{Wicket}).

\begin{figure}[h]
\centering
\includegraphics[scale=.6]{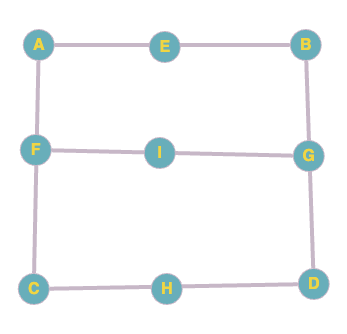}
\caption{The wicket. If we add the edge spanned by vertices $E,I,H$, then it is a grid, $GR$.}
\label{Wicket}
 \end{figure}

\medskip
In their excellent paper, titled {\em The linear Turán number of small triple systems or why is the wicket interesting?} \cite{GyS_1} they review the relevant questions and explain the importance of deciding whether $ex_L(n,W)=o(n^2)$. 

If we add the third column to the wicket to have six edges on the nine vertices, then this configuration can be avoided in dense linear hypergraphs. Let's call this hypergraph a {\em grid}, denoted by $GR$ (see in fig. \ref{Wicket}). A construction of 
Gishboliner and Shapira shows that $ex_L(n,GR)=\Omega(n^2)$ \cite{GS}. 

\section{The Main Result}

In this section, we give two proofs of the conjecture of Gy\'arf\'as and S\'ark\"ozy (Theorem \ref{Main}). In the first proof, we are using the graph regularity lemma only. The second uses hypergraph regularity. We don't know if such powerful tools are needed to prove the conjecture. Both proofs guarantee the existence of a larger graph -- containing a wicket as a subgraph -- for which a $o(n^2)$ type lower bound is known. No similar lower bound is known for wickets. A related conjecture of Gowers and Long in \cite{GL} states that there is a $c>0$ such that if no nine vertices span at least five edges in an $n$-vertex 3-uniform linear hypergraph (like in a wicket), then the number of edges is $O(n^{2-c})$. 

\begin{thm}\label{Main}
If a linear hypergraph contains no wickets, it is sparse, \\$ex_L(n,W)=o(n^2)$. 
\end{thm}

The following quantitative version of the Ruzsa-Szemer\'edi theorem \cite{RSz} is used in both proofs:
\begin{thm}\label{63density}
For every $c>0$ there exists a $\delta>0$ such that if the number of edges in a 3-uniform hypergraph $H_n^3$ is at least $cn^2$
then there exists $\delta n^3$ subgraphs $F_6^3 \subseteq H_n^3$ such that $e(F_6^3) \geq 3$.
\end{thm}

Since we are considering linear hypergraphs only, we can select a specific subgraph from $F_6^3$ if $e(F_6^3) \geq 3$. It is a 6-vertex hypergraph with three edges such that every edge has one degree-one vertex and two degree-two vertices. We call it a $(6,3)$-configuration.

The second is the ``induced matching'' variant. A matching $M$ in a graph is an induced matching if the vertices of $M$ do not span edges other than the  edges of $M$.

\begin{thm}\label{matching}
For every $c>0$ there is an $N(c)$ such that if a bipartite graph on $2n$ vertices is the union of $n$ induced matchings, and it has at least $cn^2$ edges, then $n\leq N(c)$.
\end{thm}

\subsection{First Proof of Theorem \ref{Main}}
\begin{proof}
In our first proof, we will use the above two results only. Let us suppose that for a $c>0$, we have infinitely many linear 3-uniform hypergraphs with at least $cn^2$ edges without a wicket. 
We assume that $n$ is divisible by six to avoid carrying fractions. In the proof, we will use constants, $c_i>0$, which only depend on the initial $c$.

\medskip
\noindent
We label the main steps of the proof for easier reference. 
\begin{enumerate}
    \item We can assume that the graph is three-partite with equal vertex classes $V_1,V_2,V_3$, where $|V_i|= n/3$. Theorem \ref{63density} shows that there are $c_1n^3$ $(6,3)$-configurations, each with two vertices in the classes, one of degree one and one of degree two. In a three-partite graph, the wicket looks like the graph in Fig. \ref{Wicket_drawn}.
\begin{figure}[h]
\centering
\includegraphics[scale=.6]{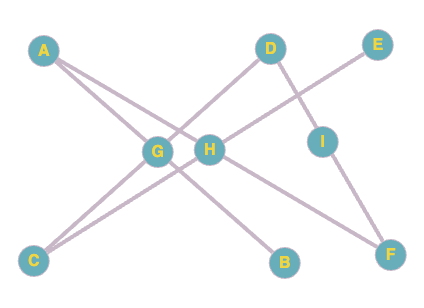}
\caption{The wicket in a 3-partite graph}
\label{Wicket_drawn}
 \end{figure}
 
    \item Take random partitions of the $V_i$-s into two equal parts. The partition classes are denoted as $V'_i,V''_i\subset V_i$ where $|V_i'|=|V_i''|=n/6$. 
    From this point, we only consider $(6,3)$-configurations which have the degree-one vertices in $V'_i$-s and the degree-two vertices in the $V_i''$-s (fig. \ref{63_1}).
    There are $c_2n^3$ such $(6,3)$-configurations.

\begin{figure}[h!]
\centering
\includegraphics[scale=.6]{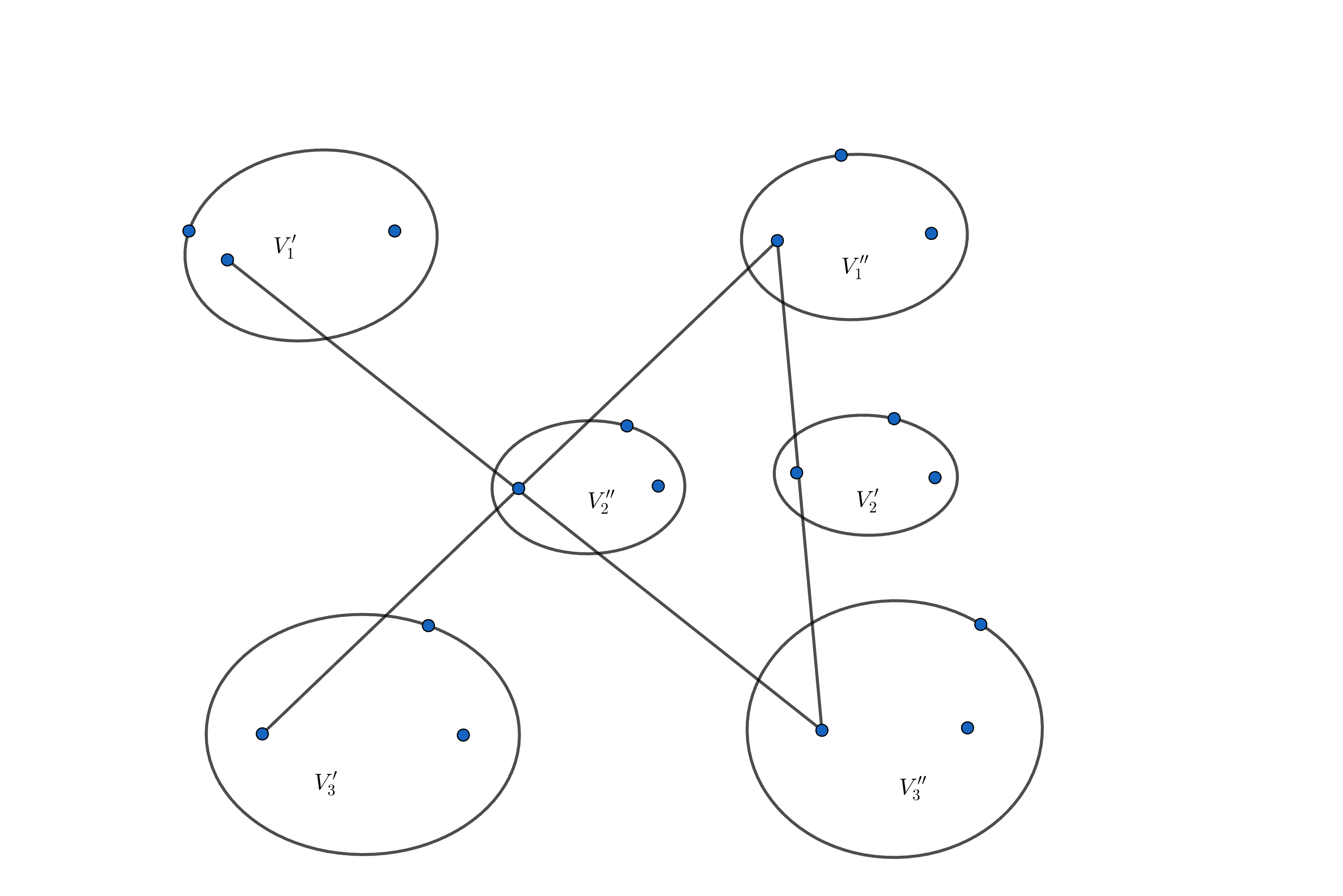}
\caption{Partitioning the vertices of the $(6,3)$-configurations}
\label{63_1}
 \end{figure}

    \item Let's select a perfect matching between the vertices of $V_1'$ and $V_3'$ out of the $(n/6)!$ possible perfect matchings at random (every matching has probability $1/(n/6)!$ to be selected). The matching is denoted by $M$, and the pairs are indexed as $M=\{m_1,m_2,\ldots,m_s\}$ where $s=n/6$. For any matching 
    between the vertices of $V_1'$ and $V_3'$ of size $k$ ($3\leq k\leq s$) the probability that at least one edge is in $M$ can be calculated by simple counting.
    The inclusion-exclusion formula gives the number of perfect matchings containing at least one edge

    \[
    k(s-1)!-\binom{k}{2}(s-2)!+\binom{k}{3}(s-3)!-\ldots \geq \frac{k(s-1)!}{2}.
    \]
    The probability that at least one edge is in $M$ is at least $k/(2s)$.
    
    \item Define an auxiliary bipartite graph $G_M$ between $V_1''$ and $V_3''$ as follows: Two vertices, $v_1\in V_1''$ and $v_3\in V_3''$ are connected by an edge if there is a pair $m_i\in M$ such that the set of the four vertices, $\{\{v_1,v_2\}\cup m_i\}$ are from a $(6,3)$-configuration.

    \item For any pair $v_1\in V_1'', v_3\in V_3''$ the probability that $(v_1,v_3)\in E(G_M)$ depends on the number of $(6,3)$-configuration containing these two vertices. This number, denoted by $M(v_1,v_3)$ is between $0$ and $s$. 
    Note that the other pairs of vertices forming a $(6,3)$-configuration with $v_1,v_3$ give a matching between the vertices of $V_1'$ and $V_3'$. 
    The probability that  $(v_1,v_3)\in E(G_M)$ is at least $M(v_1,v_3)/2s$. 
    
    By the linearity of expectations, the expected number of edges in $G_M$ is at least 
    
    \[
    \sum_{v_1\in V_1'', v_3\in V_3''}\frac{M(v_1,v_3)}{2s}={c_3n^2}.
    \]
    Let us select a matching, $M^*$ where $G_{M^*}$ has  at least ${c_3n^2}$ edges. The pairs are indexed as $M^*=\{m_1,m_2,\ldots,m_{s'}\}$.

    \item In the last step, we partition $G_{M^*}$ into matchings. For every pair $m_i\in M^*$, we select a set of edges from $G_{M^*}$ following the order of the indices of $m_i$. First, we select all $(v,w)\in E(G_{M^*})$ so that $\{\{v,w\}\cup m_1\}$ are from a $(6,3)$-configuration. The set of selected edges forms a matching, denoted by $GM_1$. We remove these edges and continue the selection of matchings. 
    In the $i$-th step, the matching $GM_i$ is given by selecting all pairs from the remaining edges 
    \[
    (v,w)\in \{G_{M^*}\setminus \{\cup_{1\leq j<i} GM_j\}\}
    \]
    where the points in $\{\{v,w\}\cup m_i\}$ are from a $(6,3)$-configuration. 
    \end{enumerate}

\medskip
\noindent
Now we are ready to complete the proof. The set of edges of $G_{M^*}$ is the union of disjoint matchings, $GM_i$ ($1\leq i\leq s'$). If $n$ is large enough, then by Theorem \ref{matching}, there is a matching, $GM_i$, which is not induced; its edges span another edge, $e$. Now, $m_i$ and the two edges of $M_i$ spanning $e$  together with $e$ span a wicket (fig. \ref{Wicket2})

\begin{figure}[h!]
\centering
\includegraphics[scale=.6]{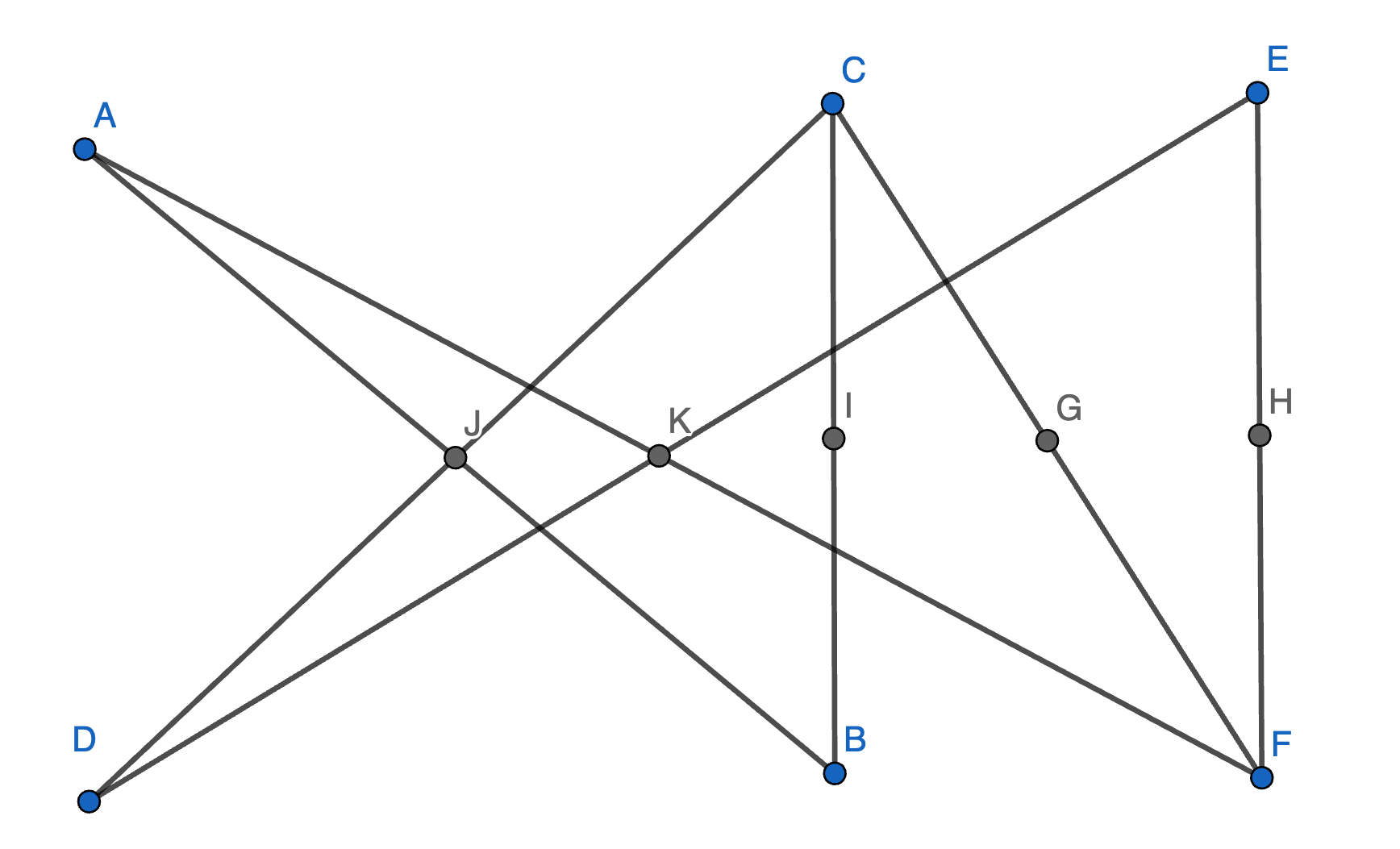}
\caption{The vertices of $m_i=\{A,D\}$, $\{(C,B),(E,F)\in GM_i\}$ and $e=(C,F)$ span a wicket with the additional vertices $J,K,G$.  }
\label{Wicket2}
 \end{figure}

\noindent
The wicket we've found is the subgraph of a larger hypergraph on 11 or 10 vertices (depending on whether $I$ and $H$ are disjoint vertices) with seven edges. Our second proof results in a larger graph on 13 or 12 vertices and nine edges.
\end{proof}

\subsection{Second Proof of Theorem \ref{Main}}

In this proof, we use the hypergraph regularity lemma by Frankl and R\"odl \cite{FR}.

\begin{thm}\label{FranklRodl}
    For any $c>0$, there is a $\delta>0$ such that the following holds for large enough $n$: If $H^3_n$ contains at least $cn^3$ pairwise edge-disjoint complete subgraphs $K^3_4$, then it has at least $\delta n^4$ complete subgraphs $K^3_4$.
\end{thm}

\begin{proof}
We follow the first two steps of the previous proof. Define a 4-partite 3-uniform hypergraph, $H^3_m$ on the vertex sets $V_1',V_1'',V_3',V_3''$. The edges are defined by the $(6,3)$-configurations; each triple of the four vertices span and edge. Every $(6,3)$-configuration spans a $K^3_4$ clique in $H^3_m$. Since every edge uniquely determines the $(6,3)$-configuration,  $H^3_m$ contains at least $c_2n^3$ pairwise edge-disjoint $K_4^3$ cliques on $m=4n/6$ vertices. By Theorem \ref{FranklRodl}, we know that the number of $K_4^3$ cliques is at least $c_4n^4$ in the graph. If two edges in a $K_4^3$ are from the same $(6,3)$-configuration, the other two are also from there since any pair of edges determines the four vertices. If $c_4n^4>c_2n^3$ then there is a $K_4^3$ where every edge is from a different $(6,3)$-configuration. Let's denote the vertices of such a $K_4^4$ by $w_1',w_1'',w_3',w_3''$, where $w_1'\in V_1',w_1''\in V_1'',w_3'\in V_3',w_3''\in V_3''$. 

\begin{figure}[h!]
\centering
\includegraphics[scale=.45]{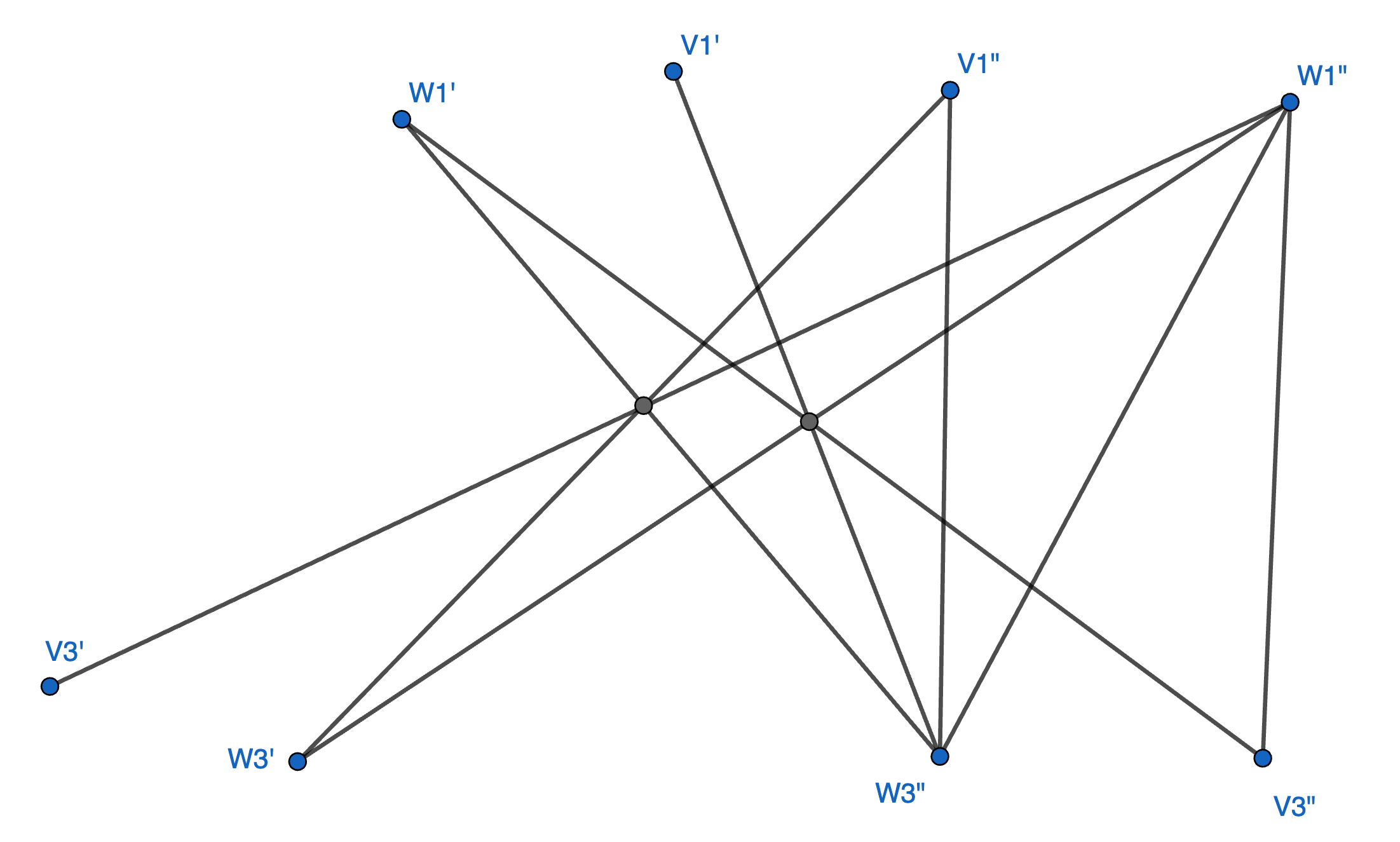}
\caption{The edges spanned by the vertices $W1',W1'', $$ W3',W3'', $$ V1'',V3''$ contain a wicket.}
\label{Wicket_B}
 \end{figure}
 
For every triple, the fourth vertex of the $(6,3)$-configuration is denoted as 
$v_1'\in V_1',v_1''\in V_1'',v_3'\in V_3',v_3''\in V_3''$. For example, the four vertices $w_1',w_1'', v_3',w_3''$ span a $(6,3)$-configuration. The eight vertices, the $w$-s and $v$-s, are all disjoint. The six vertices $w_1',w_1'',w_3',w_3'',v_1'',v_3''$ span the edges of a wicket with the two extra edges spanned by $v_1'',w_3''$ and by $w_1'',v_3''$ (fig \ref{Wicket_B}).
\end{proof}

\medskip

\section{Further Remarks}
In the two proofs, we did find two hypergraphs that are unavoidable in dense 3-uniform hypergraphs. The two graphs are shown in Fig. \ref{First} and \ref{Second}.

\begin{figure}[h]
\centering
\includegraphics[scale=.6]{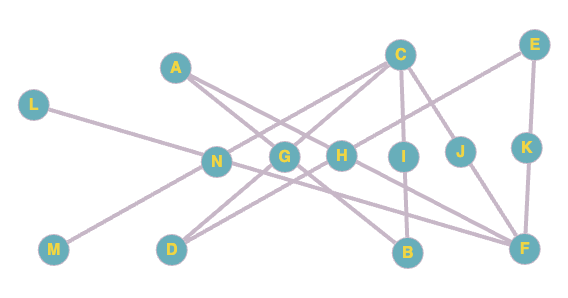}
\caption{The first unavoidable hypergraph in dense hypergraphs has 14 vertices and 9 edges, or 13 vertices and 9 edges if $I$ and $K$ are the same vertices}
\label{First}
 \end{figure}

 \begin{figure}[h]
\centering
\includegraphics[scale=.6]{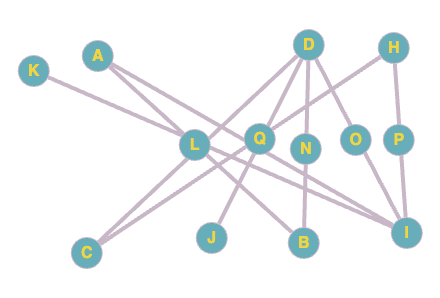}
\caption{The second unavoidable hypergraph in dense linear hypergraphs has 13 vertices and 9 edges, or 13 vertices and 9 edges if $N$ and $P$ are the same vertices}
\label{Second}
 \end{figure}

 The main question remains to decide the actual Tur\'an number of wickets in 3-uniform linear hypergraphs. Here we proved that if $H_n^{(3)}$ 
 contains no wickets, then it is sparse. On the other hand, the best-known construction avoiding wickets is an arrangement avoiding $C_4$-s, resulting in the $\Omega(n^{3/2})$ lower bound, as noted in \cite{GyS_1}.

\section{Acknowledgements}
The research was partly supported by an NSERC Discovery grant and OTKA K  grant no.119528. The author is thankful to Andr\'as Gy\'arf\'as and G\'abor S\'ark\"ozy for introducing the problem and the valuable remarks.



\begin{thebibliography}{100}


\bibitem{Conat}
David Conlon, Lior Gishboliner, Yevgeny Levanzov, Asaf Shapira,
A new bound for the Brown–Erd\H{o}s–S\'os problem,
Journal of Combinatorial Theory, Series B,
Volume 158, Part 2,
2023, Pages 1--35,

\bibitem{BES}
Paul~Erd\H{o}s William G.~Brown and Vera S\'os.
 On the existence of triangulated spheres in 3-graphs, and related
  problems.
 {\em Period. Math. Hung.}, 3:221--228.



\bibitem{FR}
Peter Frankl and Vojtech R\"odl.
\newblock Extremal problems on set systems.
\newblock {\em Random Structures and Algorithms}, 20(2):131--164, 2002.

\bibitem{GS}
Lior Gishboliner and Asaf Shapira,
Constructing dense grid-free linear 3-graphs
Proc. Am. Math. Soc., 150 (2022), pp. 69--74

\bibitem{GL}
William Gowers and Jason Long,  The length of an $s$-increasing sequence of $r$-tuples. 
Combinatorics, Probability and Computing, (2021) 30(5), 686--721.


\bibitem{GyS_1} Andr\'as Gy\'arf\'as and G\'abor N. S\'ark\"ozy,
The linear Turán number of small triple systems or why is the wicket interesting?,
Discrete Mathematics,
Volume 345, Issue 11,
2022,

\bibitem{GyS_2} Andr\'as Gy\'arf\'as and G\'abor N. S\'ark\"ozy,
Tur\'an and Ramsey numbers in linear triple systems,
Discrete Mathematics,
Volume 344, Issue 3,
2021,


\bibitem{RSz}
Imre~Z. Ruzsa and Endre Szemer\'edi.
\newblock Triple systems with no six points carrying three triangles.
\newblock In {\em Colloq. Math. Soc. J. Bolyai  Combinatorics II: 939--945, 1978}


\bibitem{SaSe}
G\'abor S\'ark\"ozy and Stanley Selkow.
\newblock An extension of the Ruzsa-Szemer\'edi theorem.
\newblock {\em Combinatorica}, 25(1):77--84, 2004.

\bibitem{SoSo}
David Solymosi and Jozsef Solymosi,
Small cores in 3-uniform hypergraphs,
Journal of Combinatorial Theory, Series B,
Volume 122,
2017,
Pages 897--910,


\end{thebibliography}
\end{document}